\documentclass[11pt]{amsart}

\usepackage{amsmath,amssymb,latexsym,soul,cite,mathrsfs}
\usepackage{tikz}
\usepackage{color,enumitem,graphicx}
\usepackage[colorlinks=true,urlcolor=blue,
citecolor=red,linkcolor=blue,linktocpage,pdfpagelabels,
bookmarksnumbered,bookmarksopen]{hyperref}
\usepackage[english]{babel}

\usepackage[left=2.7cm,right=2.7cm,top=2.8cm,bottom=2.8cm]{geometry}
\usepackage[hyperpageref]{backref}

\usepackage[colorinlistoftodos]{todonotes}
\makeatletter
\providecommand\@dotsep{5}
\def\listtodoname{List of Todos}
\def\listoftodos{\@starttoc{tdo}\listtodoname}
\makeatother

\usepackage[notref,notcite]{showkeys}

\numberwithin{equation}{section}
\newtheorem{theorem}{Theorem}[section]
\newtheorem{definition}{Definition}[section]
\newtheorem{proposition}[theorem]{Proposition}
\newtheorem{lemma}[theorem]{Lemma}
\newtheorem{corollary}[theorem]{Corollary}

\newtheorem{remark}{Remark}

\begin{document}

\title[Fractional $\kappa(\xi)$-Kirchhoff-type equation]{Existence and multiplicity of solutions for fractional $\kappa(\xi)$-Kirchhoff-type equation}

\author{J. Vanterler da C. Sousa, Kishor D. Kucche and Juan J. Nieto}

\address[J. Vanterler da C. Sousa]
{\newline\indent Aerospace Engineering, PPGEA-UEMA
\newline\indent
Department of Mathematics, DEMATI-UEMA
\newline\indent
São Luís, MA 65054, Brazil.}
\email{\href{vanterler@ime.unicamp.br}{vanterler@ime.unicamp.br}}

\address[Kishor D. Kucche ]
{\newline\indent
Department of Mathematics, Shivaji University, 
\newline\indent
Kolhapur-416 004, Maharashtra, India.}
\email{\href{kdkucche@gmail.com}{kdkucche@gmail.com}}

\address[Juan J. Nieto]
{\newline\indent
CITMAga, Departamento de Estatística, Análise Matemática e Optimización, \newline\indent Universidade de
Santiago de Compostela,  \newline\indent 15782 Santiago de Compostela, Spain}
\email{\href{juanjose.nieto.roig@usc.es}{juanjose.nieto.roig@usc.es}}

\pretolerance10000

%\begin{document}

\begin{abstract} 
In this paper, we aim to tackle the questions of existence and multiplicity of solutions to a new class of $\kappa(\xi)$-Kirchhoff-type equation utilizing a variational approach. Further, we research the results from the theory of variable exponent Sobolev spaces and from the theory of space $\psi$-fractional $\mathcal{H}^{\mu,\nu;\,\psi}_{\kappa(\xi)}(\Lambda)$. In this sense, we present a few special cases and remark on the outcomes explored.
\end{abstract}

\maketitle
%%%%%%%%%%%%%%%%%%%%%%%%%%%%%%%%%%%%%%%%%%%%%%%%%%%%%%%%%%%%%%%%%%%%%%%%%%%%%%%%%%%%%%%%%%%%%%%%%%%%%%%%%%%%%%%%%%%%%%%%%%%%%%%%%%%%%%%%%%%%%%%%%%%%%%%%%%%%%%%%%%%%
\section{Introduction and motivation}

{\color{red} In this paper, we concern the following Kirchhoff's fractional $\kappa(\xi)$-Laplacian equation
\begin{equation}\label{eq1}
\left\{ 
\begin{array}{rcl}
\mathfrak{M}\left(\displaystyle\int_{\Lambda}\frac{1}{\kappa(\xi)} \left\vert^{\rm H}\mathfrak{D}^{\mu,\nu;\,\psi}_{0+}\phi \right\vert ^{\kappa(\xi)} d \xi\right){\bf L}^{\mu,\nu;\,\psi}_{\kappa(\xi)}\phi&=&\mathfrak{g}(\xi,\phi),\,\, in\,\,\Lambda =[0,T]\times [0,T], \\ 
\phi&=&0,\,\, on\,\,\partial\Lambda 
\end{array}%
\right.
\end{equation}}
where
\begin{equation}\label{se}
    {\bf L}^{\mu,\nu;\,\psi}_{\kappa(\xi)}\phi=\,\,^{\rm H}\mathfrak{D}^{\mu,\nu;\,\psi}_{T}\left(\left\vert^{\rm H}\mathfrak{D}^{\mu,\nu;\,\psi}_{0+}\phi\right\vert^{\kappa(\xi)-2}~{^{\rm H}\mathfrak{D}^{\mu,\nu;\,\psi}_{0+}}\phi\right),
\end{equation}
$^{\rm H}\mathfrak{D}^{\mu,\nu;\,\psi}_{T}(\cdot)$ and $^{\rm H}\mathfrak{D}^{\mu,\nu;\,\psi}_{0+}(\cdot)$ {\color{red} are $\psi$-Hilfer fractional partial derivatives of order $\frac{1}{\kappa}<\mu< 1$} and type $0\leq\nu\leq 1$. Further,   $\kappa= \kappa(\xi)\in C(\bar{\Lambda})$,  $1<\kappa^{-}=\underset{\Lambda}{\inf}~ \kappa(\xi)\leq \kappa^{+}=\underset{\Lambda}{\sup}~\kappa(\xi)<2$, $\mathfrak{M}(t)$ is a continuous function and $\mathfrak{g}(\xi,\phi):\Lambda\times\mathbb{R}\rightarrow\mathbb{R}$ is the Caratheodory function.  Note that Eq.(\ref{se}), is a generalization of ${\bf L}^{\mu,\nu;\,\psi}_{\kappa}(\cdot)$ when $\kappa(\xi)=\kappa$ is a constant. 

The Kirchhoff proposed a model given by equation
\begin{equation*}
    \rho \frac{\partial^{2} u}{\partial t^{2}}- \left(\frac{\rho_{0}}{h}+ \frac{E}{2L}\int_{0}^{L} \left\vert \frac{\partial u}{\partial x}\right\vert^{2} dx\right) \frac{\partial^{2} u}{\partial x^{2}}=0,
\end{equation*}
where $\rho, ~\rho_{0}, ~L, ~h,~ E$ are constants, which extends the classical D'Alembert's wave equation.

{\color{red}The operator $$\Delta_{p(x)} u:= {\rm div}\left( |\nabla u|^{p(x)-2} |\nabla u| \right)$$ is said to be the $p(x)$-Laplacian}, and it becomes $p$-Laplacian when $p(x)=p$. The study of mathematical problems with variable exponents is very interesting. We can highlight the existence and multiplicity problem of the solution of $p(x)$-Laplacian equation, $p(x)$-Kirchhoff and $p$-Kirchhoff both in the classical and in the practical sense \cite{He,Arosio,Correa ,Correa1,Dai,Dai1,Fan,Fan2,Fan1}. See also the problems involving fractional operators and the references therein \cite{Mingqi,Pucci,Mingqi1,Pucci1}. We can also highlight fractional differential equation problems with $p$-Laplacian using variational methods, in particular, Nehari manifold \cite{Srivastava,Sousa1,Sousa80,Sousa,Sousa2,Sousa3,Ezati,Ezati1}.

In 2006, Correa and Figueiredo \cite{Correa1} investigated the existence of positive solutions to the class of problems of the $p$-Kirchhoff type
\begin{align*}
\left[-M\left(\int_{\Lambda}\left\vert \nabla u \right\vert ^{p} d x\right)\right]^{p-1} \Delta_{p} u&=f(x,u),\,\, \mbox{in}\,\,\Lambda\notag\\
u&=0,\,\,\mbox{on}\,\,\partial\Lambda,
\end{align*} 
and
\begin{align*}
\left[-M\left(\int_{\Lambda}\left\vert \nabla u \right\vert ^{p} d x\right)\right]^{p-1} \Delta_{p} u&=f(x,u)+\lambda |u|^{s-2}u,\,\, \mbox{in}\,\,\Lambda\notag\\
u&=0,\,\,\mbox{on}\,\,\partial\Lambda, 
\end{align*} 
where $\Lambda$ is a bounded smooth domain of $\mathbb{R}^{N}$, $1<p<N$, $s\geq p^{*}=\dfrac{pN}{N-p}$ and $M,f$ are continuous functions.

In 2010, Fan \cite{Fan2} considered the nonlocal $p(x)$-Laplacian Dirichlet problems with non-variational
\begin{equation*}
    -A(u) \Delta_{p(x)} u(x)= B(u) f(x,u(x))\,\,in\,\,\Lambda,\,\,u|_{\partial\Lambda}=0,
\end{equation*}
and with variational form
{\color{blue}\begin{align}\label{(222)}
-a\left(\int_{\Lambda}\frac{1}{p(x)}\left\vert \nabla u \right\vert ^{p(x)} d x\right) \Delta_{p(x)} u(x)&=b\left(\int_{\Lambda} F(x,u) dx\right)f(x,u(x))\,\, \mbox{in}\,\,\Lambda,\,\,u|_{\partial\Lambda}=0,
\end{align} }
where $$F(x,t)=\displaystyle\int_{0}^{t} f(x,s) ds, $$ and $a$ is allowed to be singular at zero. To obtain the existence and uniqueness of solutions for the problem (\ref{(222)}), the authors used variational methods, especially Mountain pass geometry.

Problems involving Kirchhoff-type with variable and non-variable exponents are attracting attention and gaining prominence in several research groups for numerous theoretical and practical questions \cite{Tang,Naimen,Dai2} and the references therein. On the other hand, it is also worth mentioning Kirchhoff's problems with fractional operators, which over the years has been increasing exponentially \cite{Xiang1,Fiscella,Ambrosio}. The $p(x)$-Laplacian {\color{blue} possesses} more complex nonlinearity which raises some of the essential difficulties, for example, it is inhomogeneous. 

{\color{blue} Dai and Hao \cite{Dai} discussed} the existence of a solution for a $p(x)$-Kirchhoff-type equation given
\begin{align*}
-M\left(\int_{\Lambda}\frac{1}{p(x)}\left\vert \nabla u \right\vert ^{p} d x\right) div \left(|\nabla|^{p(x)-2} \nabla u \right)&=f(x,u)\,\, \mbox{in}\,\,\Lambda,\notag\\
u&=0,\,\,\mbox{on}\,\,\partial\Lambda.
\end{align*} 

Motivated by the ideas found in \cite{Correa1,Fan2,Dai}, we study the existence and multiplicity of {\color{blue}solutions} for problem (\ref{eq1}) by supposing the following conditions:
\begin{itemize}
\item [($f_{0}$)] $\mathfrak{g}: \Lambda\times\mathbb{R}\rightarrow\mathbb{R}$ satisfies Caratheodory condition and
\begin{align}\label{ineq1}
	\left\vert\mathfrak{g}(\xi,t)\right\vert\leq c(1+|t|^{\tilde{\zeta}(\xi)-1}),
\end{align} 
where $\tilde{\zeta}\in C_{+}(\bar{\Lambda})$ and $\tilde{\zeta}(\xi)<\kappa^{*}_{\mu}(\xi)$ for all  $\xi\in\Lambda$.

\item [($C_{0}$)] there exists $m_{0}>0~ such ~that~ \mathfrak{M}(t)\geq m_{0}$.

\item [($C_{1}$)] there exists $  0<\omega <1$  such  that $\widehat{\mathfrak{M}}(t)\geq (1-\omega)\mathfrak{M}(t)t$, {\color{blue} where $\widehat{\mathfrak{M}}(t)=\displaystyle\int_{0}^{t} \mathfrak{M}(s)ds$.}

{\color{red}
\item [($f_{1}$)] Ambrosetti-Rabinowitz condition i.e. there exist $T>0,\,\theta > \dfrac{\kappa^{+}}{1-\omega}$ such that 
\begin{equation}\label{eq3.8}
    0<\theta ~G(\xi,t)\leq t \mathfrak{g}(\xi,t),\,\,for\,\,all\,\,|t|\geq T, a. e. \,\xi\in\Lambda,
\end{equation}
where $G(\xi,t):=\displaystyle\int_{0}^{t}\mathfrak{g}(\xi,s)ds$.}

\item [$(f_2)$] $\mathfrak{g}(\xi,t)=o(|t|^{\kappa^{+}-1}), t\to 0, \mbox{for}~ \xi\in \Lambda$ uniformly, where $\zeta^{-} > \kappa^{+}$.

\item [($f_3$)] $\mathfrak{g}(\xi,-t)= -\mathfrak{g}(\xi,t) , \xi\in \Lambda, t\in \mathbb{R}$.

\item [($f_4$)] {\color{blue}$\mathfrak{g}(\xi,t) \geq c|t|^{\gamma(\xi)-1}$,\, $t \to 0$ where} $\gamma \in C_+ (\Lambda), \kappa^{+} < \gamma ^- \leq \gamma ^+ < \dfrac{\kappa^{-}}{1-\omega}$ for a.e. $\xi\in \Lambda$.
\end{itemize}	

Our main results are the following:

\begin{theorem}\label{Teorema3.1} If $\mathfrak{M}$ satisfies {\rm($C_{0}$)} and
\begin{align}\label{ineq2}
|\mathfrak{g}(\xi,t)|\leq c(1+|t|^{\widehat{\beta} -1}),
\end{align} 
where $1\leq \widehat{\beta}<\kappa^{-}$ then problem {\rm(\ref{eq1})} has a weak solution.
\end{theorem}

\begin{theorem}\label{Teorema3.2} Assume that $\mathfrak{M}$ satisfies $\rm{(C_{0})-(C_{1})}$ and $\mathfrak{g}$ satisfies ${\normalfont(f_0),(f_1),(f_2)}$. {\color{blue} Then, problem {\rm(\ref{eq1})} has a non-trivial} solution.
\end{theorem}

\begin{theorem}\label{Teorema3.3} Assume that {\rm$(C_{0})$, $(C_{1})$, $(f_0)$} and {\rm$(f_1)$} hold and $\mathfrak{g}$ satisfies the condition {\rm$(f_3)$}. Then, problem {\rm(\ref{eq1})} {\color{blue} has a sequence of} solutions $\{\pm \phi_k\}_{k=1}^{+\infty}$ such that $\mathfrak{E}(\pm \phi_k) \to +\infty$ as $k\to +\infty$.
\end{theorem}

\begin{theorem}\label{Teorema3.4} Assume that $(C_{0})$, $(C_{1})$, $(f_0)$, $(f_1)$, $(f_2)$ $(f_3)$ hold and $\mathfrak{g}$ satisfies the condition $(f_4)$. Then, problem {\rm (\ref{eq1})} {\color{blue} has a sequence of solutions $\{\pm v_k\}_{k=1}^{+\infty}$ such that $\mathfrak{E}(\pm v_k) <0$, $\mathfrak{E}(\pm v_k) \to +\infty$ as $k\to 0$.}
\end{theorem}
The plan of the paper is as follows. In section 2, we present some definitions on fractional derivatives and integrals, among others, and results on Sobolev spaces with variable exponents and $\psi$-fractional space. In section 3, we dedicate ourselves to deal with the main contributions of the article, as highlighted above, i.e., {\bf Theorem \ref{Teorema3.1}}, {\bf Theorem \ref{Teorema3.2}}, {\bf Theorem \ref{Teorema3.3}} and {\bf Theorem \ref{Teorema3.4}}.
 
%%%%%%%%%%%%%%%%%%%%%%%%%%%%%%%%%%%%%%%%%%%%%%%%%%%%%%%%%%%%%%%%%%%%%%%%%%%%%%%%%%%%%%%%%%%%%%%%%%%%%%%%%%%%%%%%%%%%%%%%%%%%%%%%%%%%%%%%%%%%%%%%%%%%%%%%%%%%%%%%%%%%
\section{Previous results}
 In this section, we present a few essential definitions, lemmas and propositions to attack the main results of the article.

Let
$$
C_{+}(\bar{\Lambda})=\{ h~:~h\in C(\bar{\Lambda}),~h(\xi)>1 ~\mbox{for any}~ \xi\in\bar{\Lambda}\},
$$
and consider 
$$
h^{+}=\underset{\bar{\Lambda}}{\max}\,\,h(\xi),~h^{-}=\underset{\bar{\Lambda}}\,\,{\min}\,h(\xi) ~\mbox{for any }~ h\in C(\bar{\Lambda})
$$
{\color{red}and 
$$\mathscr{L}^{\kappa(\xi)}(\Lambda)=\left\{\phi\in S(\Lambda):\int_{\Lambda}\left\vert \phi(\xi)\right\vert^{\kappa(\xi)}d \xi<+\infty \right\}
$$ 
with the norm 
$$
||\phi||_{\mathscr{L}^{\kappa(\xi)}(\Lambda)}=||\phi||_{\kappa(\xi)}=\inf \left\{\lambda >0~:~\int_{\Lambda}\left|\frac{\phi(\xi)}{\lambda}\right|^{\kappa(\xi)}d \xi\leq 1 \right\},
$$
where $S(\Lambda)$ is the set of all measurable real function defined on $\Lambda$. Note that, for $\kappa(\xi)=\kappa$, we have the space $\mathscr{L}^{\kappa}$.}

The $\psi$-fractional space is given by \cite{Sousa1,Sousa80}
$$
\mathcal{H}^{\mu,\nu;\,\psi}_{\kappa(\xi)}(\Lambda)=\left\{\phi\in \mathscr{L}^{\kappa(\xi)}(\Lambda)~:~\left\vert^{\rm H}\mathfrak{D}^{\mu,\nu;\,\psi}_{0+}\phi\right\vert\in \mathscr{L}^{\kappa(\xi)}(\Lambda)\right\}
$$ 
with the norm 
$$
||\phi||=||\phi||_{\mathcal{H}^{\mu,\nu;\,\psi}_{\kappa(\xi)}(\Lambda)}=||\phi||_{\mathscr{L}^{\kappa(\xi)}(\Lambda)}+\left\|^{\rm H}\mathfrak{D}^{\mu,\nu;\,\psi}_{0+}\phi\right\|_{\mathscr{L}^{\kappa(\xi)}(\Lambda)}.
$$ 
{\color{red}Denote by $\mathcal{H}^{\mu,\nu;\,\psi}_{\kappa(\xi),0}(\Lambda)$ the closure of $C_{0}^{\infty}(\Lambda)$ in $\mathcal{H}^{\mu,\nu;\,\psi}_{\kappa(\xi)}(\Lambda).$}

Next, we will present the definitions of Riemann-Liouville partial fractional integrals with respect to another function and of the fractional derivatives $\psi$-Hilfer for $3$-variables. For a study of $N$-variables, see \cite{Srivastava,J1}.

Let $\theta=(\theta_{1},\theta_{2},\theta_{3})$, $T=(T_{1},T_{2},T_{3})$ and $\mu=(\mu_{1},\mu_{2},\mu_{3})$ where $0<\mu_{1},\mu_{2},\mu_{3}<1$ with $\theta_{j}<T_{j}$, for all $j\in \left\{1,2,3 \right\}$. {\color{red}Also put $\Lambda=I_{1}\times I_{2}\times I_{3}=[\theta_{1},T_{1}]\times [\theta_{2},T_{2}]\times [\theta_{3},T_{3}]$, where $T_{1},T_{2},T_{3}$ and $\theta_{1},\theta_{2},\theta_{3}$ are positive constants}. Consider also $\psi(\cdot)$ be an increasing and positive monotone function on $(\theta_{1},T_{1}),(\theta_{2},T_{2}),(\theta_{3},T_{3})$, having a continuous derivative $\psi'(\cdot)$ on $(\theta_{1},T_{1}],(\theta_{2},T_{2}],(\theta_{3},T_{3}]$. The $\psi$-Riemann-Liouville fractional partial integrals of $\phi\in \mathscr{L}^{1}(\Lambda)$ of order $\mu$ $(0<\mu<1)$ are given by \cite{Srivastava,J1}:
\begin{itemize}
    \item 1-variable: right and left-sided
\begin{equation*}
    {\bf I}^{\mu,\psi}_{\theta_{1}} \phi(\xi_{1})=\dfrac{1}{\Gamma(\mu)} \int_{\theta_{1}}^{\xi_{1}} \psi'(s_{1})(\psi(\xi_{1})- \psi(s_{1}))^{\mu-1} \phi(s_{1}) ds_{1},\,\,to\,\,\theta_{1}<s_{1}<\xi_{1}
\end{equation*}
and
\begin{equation*}
    {\bf I}^{\mu,\psi}_{T_{1}} \phi(\xi_{1})=\dfrac{1}{\Gamma(\mu)} \int_{\xi_{1}}^{T_{1}} \psi'(s_{1})(\psi(s_{1})- \psi(\xi_{1}))^{\mu-1} \phi(s_{1}) ds_{1},\,\,to\,\,\xi_{1}<s_{1}<T_{1},
\end{equation*}
with $\xi_{1}\in[\theta_{1},T_{1}]$, respectively.
    
\item 3-variables: right and left-sided
\begin{eqnarray*}
    {\bf I}^{\mu,\psi}_{\theta} \phi(\xi_{1},\xi_{2},\xi_{3})=\dfrac{1}{\Gamma(\mu)\Gamma(\mu_{2})\Gamma(\mu_{3})} \int_{\theta_{1}}^{\xi_{1}}
    \int_{\theta_{2}}^{\xi_{2}}
    \int_{\theta_{3}}^{\xi_{3}}
    \psi'(s_{1})\psi'(s_{2})\psi'(s_{3})
    (\psi(\xi_{1})- \psi(s_{1}))^{\mu_1-1}\notag\\
    \times
    (\psi(\xi_{2})- \psi(s_{2}))^{\mu_{2}-1}
    (\psi(\xi_{3})- \psi(s_{3}))^{\mu_{3}-1}
    \phi(s_{1},s_{2},s_{3}) ds_{3}ds_{2}ds_{1},
\end{eqnarray*}
to $\theta_{1}<s_{1}<\xi_{1}, \theta_{2}<s_{2}<\xi_{2}, \theta_{3}<s_{3}<\xi_{3}$ and
\begin{eqnarray*}
    {\bf I}^{\mu,\psi}_{T} \phi(\xi_{1},\xi_{2},\xi_{3})=\dfrac{1}{\Gamma(\mu)\Gamma(\mu_{2})\Gamma(\mu_{3})} \int_{\xi_{1}}^{T_{1}}
    \int_{\xi_{2}}^{T_{2}}
    \int_{\xi_{3}}^{T_{3}}
    \psi'(s_{1})\psi'(s_{2})\psi'(s_{3})
    (\psi(s_{1})-\psi(\xi_{1}))^{\mu_1-1}\notag\\
    \times
    (\psi(s_{2})-\psi(\xi_{2}))^{\mu_{2}-1}
    (\psi(s_{3})-\psi(\xi_{3}))^{\mu_{3}-1}
    \phi(s_{1},s_{2},s_{3}) ds_{3}ds_{2}ds_{1},
\end{eqnarray*}
with $\xi_{1}<s_{1}<T_{1}, \xi_{2}<s_{2}<T_{2}, \xi_{3}<s_{3}<T_{3}$, $\xi_{1}\in[\theta_{1},T_{1}]$, $\xi_{2}\in[\theta_{2},T_{2}]$ and $\xi_{3}\in[\theta_{3},T_{3}]$, respectively.
\end{itemize}

On the other hand, let $\phi,\psi \in C^{n}(\Lambda)$ {\color{blue} be two} functions such that $\psi$ is increasing and $\psi'(\xi_{j})\neq 0$ with $\xi_{j}\in[\theta_{j},T_{j}]$, $j\in \left\{1,2,3 \right\}$. The left and
right-sided $\psi$-Hilfer fractional partial derivative of $3$-variables of $\phi\in AC^{n}(\Lambda)$ of order $\mu=(\mu_{1},\mu_{2},\mu_{3})$ $(0<\mu_{1},\mu_{2},\mu_{3}\leq 1)$ and type $\nu=(\nu_{1},\nu_{2},\nu_{3})$ where $0\leq\nu_{1},\nu_{2},\nu_{3}\leq 1$, are defined by \cite{Srivastava,J1}
\begin{equation}\label{derivada1}
{^{\mathbf H}\mathfrak{D}}^{\mu,\nu;\psi}_{\theta}\phi(\xi_{1},\xi_{2},\xi_{3})= {\bf I}^{\nu(1-\mu),\psi}_{\theta} \Bigg(\frac{1}{\psi'(\xi_{1})\psi'(\xi_{2})\psi'(\xi_{3})} \Bigg(\frac{\partial^{3}} {\partial \xi_{1}\partial \xi_{2}\partial \xi_{3}}\Bigg) \Bigg) {\bf I}^{(1-\nu)(1-\mu),\psi}_{\theta} \phi(\xi_{1},\xi_{2},\xi_{3})
\end{equation}
and
\begin{equation}\label{derivada2}
{^{\mathbf H}\mathfrak{D}}^{\mu,\nu;\psi}_{T}\phi(\xi_{1},\xi_{2},\xi_{3})= {\bf I}^{\nu(1-\mu),\psi}_{T} \Bigg(-\frac{1}{\psi'(\xi_{1})\psi'(\xi_{2})\psi'(\xi_{3})} \Bigg(\frac{\partial^{3}} {\partial \xi_{1}\partial \xi_{2}\partial \xi_{3}}\Bigg) \Bigg) {\bf I}^{(1-\nu)(1-\mu),\psi}_{T} \phi(\xi_{1},\xi_{2},\xi_{3}),
\end{equation}
where $\theta$ and $T$ are the same parameters presented in the definition of fractional integrals ${\bf I}_{T}^{\mu;\psi}(\cdot)$ and ${\bf I}_ {\theta}^{\mu;\psi}(\cdot)$.

Taking $\theta =0$ in the definition of ${^{\mathbf H}\mathfrak{D}}^{\mu,\nu;\psi}_{\theta}(\cdot)$, we have ${^{\mathbf H}\mathfrak{D}}^{\mu,\nu;\psi}_{0}(\cdot)$. During the paper we will use the following notation: ${^{\mathbf H}\mathfrak{D}}^{\mu,\nu;\psi}_{\theta} \phi(\xi_{1},\xi_{2},\xi_{3}):= {^{\mathbf H}\mathfrak{D}}^{\mu,\nu;\psi}_{\theta} \phi$, ${^{\mathbf H}\mathfrak{D}}^{\mu,\nu;\psi}_{T} \phi(\xi_{1},\xi_{2},\xi_{3}):= {^{\mathbf H}\mathfrak{D}}^{\mu,\nu;\psi}_{T} \phi$ and ${\bf I}_{\theta}^{\mu;\psi}\phi(\xi_{1},\xi_{2},\xi_{3}):= {\bf I}_{\theta}^{\mu;\psi}\phi$.

{\color{red}Let $\theta=(\theta_{1},\theta_{2})$, $T=(T_{1},T_{2})$ and $\mu=(\mu_{1},\mu_{2})$. The relation
\begin{equation}\label{eq.218}
\int_{\theta_{1}}^{T_1}\int_{\theta_2}^{T_2}\left( {\bf I}_{\theta}^{\mu ;\psi }\varphi\left( \xi_{1},\xi_{2}\right) \right) \phi\left( \xi_{1},\xi_{2}\right) {\rm d\xi_{2} d\xi_{1}}=\int_{\theta_{1}}^{T_1}\int_{\theta_2}^{T_2}\varphi\left( \xi_{1},\xi_{2}\right) \psi ^{\prime }\left( \xi_{1}\right) \psi'(\xi_{2}) {\bf I}_{T}^{\mu ;\psi }\left( \frac{\phi\left( \xi_{1},\xi_{2}\right) }{\psi ^{\prime }\left( \xi_{1}\right)\psi'(\xi_{2}) }\right) {\rm d\xi_{2} d\xi_{1}}
\end{equation}%
is valid.}

{\color{red}One can prove Eq.(\ref{eq.218}) directly by interchanging the order of integration by the Dirichlet formula in the particular case Fubini theorem, i.e.,
\begin{eqnarray*}
&&\int_{\theta_{1}}^{T_1}\int_{\theta_2}^{T_2}\left( {\bf I}_{\theta}^{\mu ;\psi }\varphi\left( \xi_{1},\xi_{2}\right) \right) \phi\left( \xi_{1},\xi_{2}\right) {\rm d\xi_{2} d\xi_{1}}\notag\\ &=&\int_{\theta_{1}}^{T_1}\int_{\theta_2}^{T_2}\frac{1}{\Gamma \left( \mu_{1} \right) \Gamma \left( \mu_{2} \right)}
\int_{\theta_{1}}^{\xi_{1}}\int_{\theta_{2}}^{\xi_{2}}\psi ^{\prime }\left( s_{1}\right) \psi ^{\prime }\left( s_{2}\right) \left( \psi \left( \xi_{1}\right) -\psi \left( s_{1}\right) \right) ^{\mu_{1} -1} 
\left( \psi \left( \xi_{2}\right) -\psi \left( s_{2}\right) \right) ^{\mu_{2} -1}\notag\\&& \times\varphi\left( s_{1},s_{2}\right) {\rm ds_{2} ds_{1}}\phi\left( \xi_{1},\xi_{2}\right) {\rm d\xi_{2} d\xi_{1}}
\notag \\
&=&\int_{\theta_{1}}^{T_1}\int_{\theta_2}^{T_2}\frac{1}{\Gamma \left( \mu_{1} \right) \Gamma \left( \mu_{2} \right)}
\int_{\xi_{1}}^{T_{1}}\int_{\xi_{2}}^{T_{2}}\psi ^{\prime }\left( s_{1}\right) \psi ^{\prime }\left( s_{2}\right) \left( \psi \left( \xi_{1}\right) -\psi \left( s_{1}\right) \right) ^{\mu_{1} -1} 
\left( \psi \left( \xi_{2}\right) -\psi \left( s_{2}\right) \right) ^{\mu_{2} -1}\notag\\&& \times\phi\left( \xi_{1},\xi_{2}\right) {\rm d\xi_{2} d\xi_{1}}\varphi\left( s_{1},s_{2}\right) {\rm ds_{2} ds_{1}}
\notag \\
&=&\int_{\theta_{1}}^{T_1}\int_{\theta_2}^{T_2}\psi'(s_{1})\psi'(s_{2}) \varphi(s_{1},s_{2}){\bf I}_{T}^{\mu ;\psi }\left( \frac{\phi\left( s_{1},s_{2}\right) }{\psi ^{\prime }\left(s_{1}\right) \psi'(s_{2}) }\right) {\rm ds_{2} ds_{1}}.
\end{eqnarray*}}

{\color{red}
\begin{theorem} Let  $\psi(\cdot)$ be an increasing and positive monotone function on $[\theta_{1},T_1]\times [\theta_{2},T_2]$, having a continuous derivative $\psi'(\cdot)\neq 0$ on $(\theta_1,T_1)\times(\theta_2,T_2)$. If $0<\mu=(\mu_{1},\mu_{2}) <1$ and $0\leq \nu=(\nu_{1},\nu_{2}) \leq 1$, then
\begin{eqnarray}\label{mera}
&&\int_{\theta_{1}}^{T_1}\int_{\theta_2}^{T_2}\left( ^{{\bf H}}\mathfrak{D}_{\theta}^{\mu,\nu ;\psi }\varphi\left( \xi_{1},\xi_{2}\right) \right) \phi\left( \xi_{1},\xi_{2}\right) {\rm d\xi_{2} d\xi_{1}}\notag\\&&=\int_{\theta_{1}}^{T_1}\int_{\theta_2}^{T_2}\varphi\left( \xi_{1},\xi_{2}\right) \psi ^{\prime }\left(\xi_{1}\right) \psi ^{\prime }\left(\xi_{2}\right) \text{ }^{{\bf H}}\mathfrak{D}_{T}^{\mu ,\nu ;\psi }\left( \frac{\phi\left( \xi_{1},\xi_{2}\right) }{\psi ^{\prime }\left( \xi_{1}\right)\psi ^{\prime }\left( \xi_{2}\right) }\right) {\rm d\xi_{2}d\xi_{1}}
\end{eqnarray}
for any $\varphi\in C^{1}$ and $\phi\in C^{1}$ satisfying the boundary conditions $\varphi\left( \theta_1,\theta_2\right)=0=\varphi\left( T_1,T_2\right)$.
\end{theorem}}

{\color{red}
\begin{proof} In fact, using the Eq.(\ref{eq.218}), one has
\begin{eqnarray*}
&&\int_{\theta_{1}}^{T_1}\int_{\theta_2}^{T_2}\varphi\left( \xi_{1},\xi_{2}\right) \psi ^{\prime }\left( \xi_{1}\right)\psi ^{\prime }\left( \xi_{2}\right) \text{ }^{{\bf H}}\mathfrak{D}_{T}^{\mu ,\nu ;\psi }\left( \frac{\phi\left( \xi_{1},\xi_{2}\right) }{\psi^{\prime }\left( \xi_{1}\right) \psi'(\xi_{2})}\right) {\rm d\xi_{2}d\xi_{1}}  \notag \\
&=&\int_{\theta_{1}}^{T_1}\int_{\theta_2}^{T_2}\varphi\left( \xi_{1},\xi_{2}\right) \psi ^{\prime }\left( \xi_{1}\right)\psi'(\xi_{2})
{\bf I}_{T}^{\gamma-\mu ;\psi }\text{ }D_{T}^{\gamma;\psi }\left( \frac{\phi\left(
\xi_{1},\xi_{2}\right) }{\psi ^{\prime }\left(\xi_{1}\right)\psi'(\xi_{2}) }\right) {\rm d\xi_{2}d\xi_{1}}  \notag \\
&=&\int_{\theta_{1}}^{T_1}\int_{\theta_2}^{T_2}\psi ^{\prime }\left( \xi_{1}\right)\psi'(\xi_{2}) \left[ {\bf I}_{\theta}^{\mu ;\psi }  \text{ }^{{\bf H}}\mathfrak{D}_{\theta}^{\mu ,\nu ;\psi }\varphi\left( \xi_{1},\xi_{2}\right) +\frac{\left( \psi \left( \xi_{1}\right) -\psi \left( \theta_1\right) \right) ^{\gamma -1}\left( \psi \left( \xi_{2}\right) -\psi \left( \theta_2\right) \right) ^{\gamma -1}
}{\Gamma \left(\gamma \right) }d_{j}\right]\notag\\&&\times {\bf I}_{T}^{\gamma-\mu ;\psi }\text{ }D_{T}^{\gamma;\psi
}\left( \frac{\phi\left( \xi_{1},\xi_{2}\right) }{\psi ^{\prime }\left( \xi_{1}\right)\psi'(\xi_{2}) }\right) {\rm d\xi_{1} d\xi_{2}}
\text{ }\notag\\&& \left( \text{where }d_{j}=\left( \frac{1}{\psi ^{\prime }\left(\xi_{1}\right) \psi'(\xi_{2}) }\frac{d}{{\rm d\xi_{1}}}\frac{d}{{\rm d\xi_{2}}}\right) {\bf I}_{\theta}^{\left( 1-\nu \right) \left( 1-\mu
\right) ;\psi }\varphi\left( \theta_1,\theta_2\right) \right)  \notag \\
&=&\int_{\theta_{1}}^{T_1}\int_{\theta_2}^{T_2}\psi ^{\prime }\left( \xi_{1}\right) \psi'(\xi_{2}) {\bf I}_{\theta}^{\mu ;\psi }\text{ }
^{{\bf H}}\mathfrak{D}_{\theta}^{\mu ,\nu ;\psi }\varphi\left( \xi_{1},\xi_{2}\right) {\bf I}_{T}^{\gamma-\mu ;\psi }
\text{ }D_{T}^{\gamma;\psi }\left( \frac{\phi\left( \xi_{1},\xi_{2}\right) }{\psi ^{\prime
}\left( \xi_{1}\right) \psi'(\xi_{2}) }\right) {\rm d\xi_{2}d\xi_{1}}\notag\\&+&\frac{d_{j}}{\Gamma \left( \gamma \right) }\int_{\theta_{1}}^{T_1}\int_{\theta_2}^{T_2}\psi ^{\prime }\left( \xi_{1}\right) \psi'(\xi_{2}) \left( \psi \left( \xi_{1}\right) -\psi \left( \theta_1\right) \right) ^{\gamma -1}\left( \psi \left( \xi_{2}\right) -\psi \left( \theta_2\right) \right) ^{\gamma -1}\notag\\&& \times{\bf I}_{b-}^{\gamma-\mu ;\psi }\text{ } D_{T}^{\gamma;\psi }\left( \frac{\phi\left( \xi_{1},\xi_{2}\right) }{\psi ^{\prime }\left(\xi_{1}\right)\psi'(\xi_{2}) }\right) {\rm d\xi_{2}d\xi_{1}}  \notag \\
&=&\int_{\theta_{1}}^{T_1}\int_{\theta_2}^{T_2}{\bf I}_{\theta}^{\mu ;\psi }\text{ }^{{\bf H}}\mathfrak{D}_{\theta}^{\mu ,\nu ;\psi }\varphi\left( \xi_{1},\xi_{2}\right) {\bf I}_{T}^{-\mu ;\psi }\left( \frac{\phi\left( \xi_{1},\xi_{2}\right) }{ \psi ^{\prime }\left( \xi_{1}\right)\psi'(\xi_{2}) }\right) {\rm d\xi_{2}d\xi_{1}}  \notag \\
&=&\int_{\theta_{1}}^{T_1}\int_{\theta_2}^{T_2}\text{ }\left( ^{{\bf H}}\mathfrak{D}_{\theta}^{\mu ,\nu ;\psi }\varphi\left( \xi_{1},\xi_{2}\right) \right) \phi\left( \xi_{1},\xi_{2}\right) {\rm d\xi_{2}d\xi_{1}},
\end{eqnarray*}
where $D_{T}^{\gamma;\psi }(\cdot)$ is the $\psi$-Riemann-Liouville fractional derivative with $\gamma=\mu+\nu(1-\mu)$.
\end{proof}}

{\color{red}
}
\begin{proposition}\label{bana} {\rm\cite{Srivastava}} The spaces $\mathscr{L}^{\kappa(\xi)}(\Lambda)$ and $\mathcal{H}^{\mu,\nu;\,\psi}_{\kappa(\xi)}(\Lambda)$ are separable and reflexive Banach spaces.
\end{proposition}

\begin{proposition}{\rm \cite{20}} Set $\rho(\phi)=\displaystyle\int_{\Lambda}|\phi(\xi)|^{\kappa(\xi)}d \xi$. For any $\phi\in \mathscr{L}^{\kappa(\xi)}(\Lambda)$. Then,
\begin{itemize}

\item [1)] For $\phi\neq 0,~|\phi|_{\kappa(\xi)}=\lambda$ if and only if $\rho\left(\frac{\phi}{\lambda}\right)=1,$

\item [2)] $|\phi|_{\kappa(\xi)}<1,~(=1;>1)$ if and only if $\rho(\phi)<1~(=1,>1),$

\item [3)] If $|\phi|_{\kappa(\xi)}>1,$ then $|\phi|^{\kappa^{-}}_{\kappa(\xi)}\leq \rho(\phi)\leq |\phi|^{\kappa^{+}}_{\kappa(\xi)},$

\item [4)] If $|\phi|_{\kappa(\xi)}<1,$ then $|\phi|^{\kappa^{+}}_{\kappa(\xi)}\leq \rho(\phi)\leq |\phi|^{\kappa^{-}}_{\kappa(\xi)},$

\item [5)] $\lim\limits_{k\rightarrow +\infty}|\phi_{k}|_{\kappa(\xi)}=0$ if and only if $\lim\limits_{k\rightarrow +\infty}\rho(\phi_{k})=0$

\item [6)] $\lim\limits_{k\rightarrow +\infty}|\phi_{k}|_{\kappa(\xi)}=+\infty$ if and only if $\lim\limits_{k\rightarrow +\infty}\rho(\phi_{k})=+\infty$.
\end{itemize}	
\end{proposition}

\begin{proposition}{\color{red}{}\rm \cite{Fan12}\label{proposition2.4} If $\phi,\phi_{k}\in \mathscr{L}^{\kappa(\xi)}(\Omega)$, $k=1,2,...$ then the following statements are equivalent each other}
\begin{enumerate}
    \item {\color{red}$\lim_{k\rightarrow+\infty} |\phi_{k}-\phi|_{\kappa(\xi)}=0$; }
    
    \item $\lim_{k\rightarrow+\infty} \rho(\phi_{k}-\phi)=0$;
    
    \item $\phi_{k}\rightarrow\phi$ in measure in $\Omega$ and $\lim_{k\rightarrow+\infty}\rho(\phi_{k})=\rho(\phi)$.
\end{enumerate}
\end{proposition}

\begin{proposition}{\rm \cite{Sousa80,Sousa}} For any $\phi\in \mathcal{H}^{\mu,\nu;\,\psi}_{\kappa(\xi)}(\Lambda)$ there exists a positive constant $c$ such that $$||\phi||_{\mathscr{L}^{\kappa(\xi)}(\Lambda)}\leq c\, \left\|^{\rm H}\mathfrak{D}^{\mu,\nu;\,\psi}_{0+}\phi\right\|_{\mathscr{L}^{\kappa(\xi)}(\Lambda)}.$$
\end{proposition}

{\color{red} In this sense, we have that the norms $||\phi||$ and $\left\|^{\rm H}\mathfrak{D}^{\mu,\nu;\,\psi}_{0+}\phi \right\|_{\mathscr{L}^{\kappa(\xi)}(\Lambda)}$ are equivalent in the space $\mathcal{H}^{\mu,\nu;\psi}_{\kappa(\xi)}(\Lambda)$, so let's use $||\phi||=\left\|^{\rm H}\mathfrak{D}^{\mu,\nu;\,\psi}_{0+}\phi\right\|_{\mathscr{L}^{\kappa(\xi)}(\Lambda)}$, for simplicity {\rm \cite{Sousa80,Sousa}}.}

\begin{proposition}{\rm\cite{Srivastava}}\label{p1} Assume that the boundary of $\Lambda$ possess the property $\kappa\in C(\bar{\Lambda})$ with $\kappa(\xi)<2.$ If $q\in C({\bar{\Lambda}}) $ and $1\leq h(\xi)\leq \kappa^{*}_{\mu}(\xi),~ (1\leq q(\xi)<\kappa^{*}_{\mu}(\xi))$ for $\xi\in\Lambda$ then there is a continuous (compact) embedding $ \mathcal{H}^{\mu,\nu;\,\psi}_{\kappa(\xi)}(\Lambda) \hookrightarrow\mathscr{L}^{q(\xi)}(\Lambda),$ whose $\kappa^{*}_\mu =\dfrac{2 \kappa }{2-\mu {\kappa}}$. 
\end{proposition}

We write 
	\begin{equation*}
	    \mathcal{I}^{\mu,\nu}(\phi)=\int_{\Lambda}\dfrac{1}{\kappa(\xi)}\left\vert^{\rm H}\mathfrak{D}^{\mu,\nu;\,\psi}_{0+}\phi\right\vert^{\kappa(\xi)}d \xi
	\end{equation*}
where meas $\left\{ \xi\in \Lambda,\,\phi^{+}>\theta\right\}>0$. We denote $\mathcal{L}^{\mu,\nu}=(\mathcal{I}^{\mu,\nu})':\mathcal{H}^{\mu, \nu;\psi}_{\kappa(\xi)}(\Lambda)\rightarrow \left(\mathcal{H}^{\mu, \nu;\psi}_{\kappa(\xi)}(\Lambda)\right)^{*}$ then
\begin{equation*}
(\mathcal{L}^{\mu,\nu} (\phi,v))=\int_{\Lambda} \left\vert^{\rm H}\mathfrak{D}^{\mu,\,\nu;\,\psi}_{0+}\phi\right\vert^{\kappa(\xi)-2}\,\, ^{\rm H}\mathfrak{D}^{\mu,\,\nu;\,\psi}_{0+}\phi\,\, ^{\rm H}\mathfrak{D}^{\mu,\,\nu;\,\psi}_{0+}v\,\, d\xi
\end{equation*}
for all $\phi,v\in \mathcal{H}^{\mu, \nu;\psi}_{\kappa(\xi)}(\Lambda)$.

\begin{proposition}\label{p2} 1.$\mathcal{L}^{\mu,\nu}:\mathcal{H}^{\mu, \nu;\psi}_{\kappa(\xi)}(\Lambda)\rightarrow \left(\mathcal{H}^{\mu, \nu;\psi}_{\kappa(\xi)}(\Lambda)\right)^{*}$ is a continuous, bounded and strictly monotone operator;
\\
2. $\mathcal{L}^{\mu,\nu}$ is a mapping of type $(S_{+})$, i.e., if $\phi_{n}\rightharpoonup \phi$ in $\mathcal{H}^{\mu, \nu;\psi}_{\kappa(\xi)}(\Lambda)$ and $\overline{\lim}_{n\rightarrow+\infty} (\mathcal{L}^{\mu,\nu}(\phi_{n})-\mathcal{L}^{\mu,\nu}(\phi), \phi_{n}-\phi)\leq 0$, then $\phi_{n}\rightarrow \phi$ in $\mathcal{H}^{\mu, \nu;\psi}_{\kappa(\xi)}(\Lambda)$;
\\
3. $\mathcal{L}^{\mu,\nu}:\mathcal{H}^{\mu, \nu;\psi}_{\kappa(\xi)}(\Lambda)\rightarrow\left(\mathcal{H}^{\mu, \nu;\psi}_{\kappa(\xi)}(\Lambda)\right)^{*}$ is a homeomorphism.
\end{proposition}

\begin{proof} 1. It is obvious that $\mathcal{L}^{\mu,\nu}$ is continuous and bounded. {\color{blue} For any $\xi,y\in \Lambda$ \cite{Fan1,AN}}
\begin{equation}\label{(1)}
    \left[\left( |\xi|^{\kappa-2} \xi- |y|^{\kappa-2} y)\right (\xi-y) \right] \left( |\xi|^{\kappa}-|y|^{\kappa} \right)^{(2-\kappa)/\kappa} \geq (\kappa-1)|\xi-y|^{\kappa}
\end{equation}
with $1<\kappa<2$ and
\begin{equation}\label{(2)}
    \left( |\xi|^{\kappa-2} \xi- |y|^{\kappa-2} y)\right (\xi-y) \geq \left(\frac{1}{2}\right)^{\kappa}|\xi-y|^{\kappa}, \,\, \kappa\geq 2.
\end{equation}

2. {\color{red}From inequality (\ref{(1)}), if $\phi_{n}\rightharpoonup \phi$ and $\overline{\lim}_{n\rightarrow+\infty} (\mathcal{L}^{\mu,\nu}(\phi_{n})- \mathcal{L}^{\mu,\nu}(\phi),\phi_{n}-\phi )\leq 0$, then $\lim_{n\rightarrow+\infty} \left(\mathcal{L}^{\mu,\nu}(\phi_{n})-\mathcal{L}^{\mu,\nu}(\phi),\phi_{n}-\phi \right)=0$. In view of inequalities (\ref{(1)}) and (\ref{(2)}), 
$^{\rm H}\mathfrak{D}^{\mu,\,\nu;\,\psi}_{0+}\phi_{n}$ goes in measure to $^{\rm H}\mathfrak{D}^{\mu,\,\nu;\,\psi}_{0+}\phi$ in $\Lambda$, so we get a subsequence, satisfying $^{\rm H}\mathfrak{D}^{\mu,\,\nu;\,\psi}_{0+}\phi_{n}\rightarrow\,\, ^{\rm H}\mathfrak{D}^{\mu,\,\nu;\,\psi}_{0+}\phi$ a.e. $\xi\in\Lambda$. Using {\color{blue}Fatou's} lemma, yields
\begin{equation}\label{(3)}
    \underline{\lim}_{n\rightarrow+\infty} \int_{\Lambda} \frac{1}{\kappa(\xi)} \, \left\vert^{\rm H}\mathfrak{D}^{\mu,\,\nu;\,\psi}_{0+}\phi_{n}\right\vert^{\kappa(\xi)} d\xi \geq \int_{\Lambda} \frac{1}{\kappa(\xi)} \, \left\vert^{\rm H}\mathfrak{D}^{\mu,\,\nu;\,\psi}_{0+}\phi \right\vert^{\kappa(\xi)} d\xi.
\end{equation}}

{\color{red}
From $\phi_{n}\rightharpoonup \phi$, yields
\begin{equation}\label{(4)}
\lim_{n\rightarrow+\infty} (\mathcal{L}^{\mu,\nu} (\phi_{n}),\phi_{n}-\phi )=\, \lim_{n\rightarrow+\infty} (\mathcal{L}^{\mu,\nu} (\phi_{n}) -\mathcal{L}^{\mu,\nu} (\phi), \phi_{n}-\phi )=0.
\end{equation}}

{\color{red} On the other hand, {\color{blue}we also have}
\begin{eqnarray}\label{(5)}
    &&\left( \mathcal{L}^{\mu,\nu}(\phi_{n}),\phi_{n}-\phi \right)\notag\\&=&\int_{\Lambda} \left\vert^{\rm H}\mathfrak{D}^{\mu,\,\nu;\,\psi}_{0+}\phi_{n}\right\vert^{\kappa(\xi)-2}\,\, ^{\rm H}\mathfrak{D}^{\mu,\,\nu;\,\psi}_{0+}\phi_{n}\,\, ^{\rm H}\mathfrak{D}^{\mu,\,\nu;\,\psi}_{0+}(\phi_{n}-\phi) d\xi\notag\\
    &=& \int_{\Lambda}  \left\vert^{\rm H}\mathfrak{D}^{\mu,\,\nu;\,\psi}_{0+}\phi_{n}\right\vert^{\kappa(\xi)} d\xi - \, \int_{\Lambda} \left\vert^{\rm H}\mathfrak{D}^{\mu,\,\nu;\,\psi}_{0+}\phi_{n}\right\vert^{\kappa(\xi)-2}\,\, ^{\rm H}\mathfrak{D}^{\mu,\,\nu;\,\psi}_{0+}\phi_{n}\,\, ^{\rm H}\mathfrak{D}^{\mu,\,\nu;\,\psi}_{0+}\phi d\xi\notag \\
    &\geq & \int_{\Lambda} \left\vert^{\rm H}\mathfrak{D}^{\mu,\,\nu;\,\psi}_{0+}\phi_{n}\right\vert^{\kappa(\xi)} d\xi - \int_{\Lambda} \left( \frac{\kappa(\xi)-1}{\kappa(\xi)}  \left\vert^{\rm H}\mathfrak{D}^{\mu,\,\nu;\,\psi}_{0+}\phi_{n}\right\vert^{\kappa(\xi)} d\xi+ \frac{1}{\kappa(\xi)} \left\vert^{\rm H}\mathfrak{D}^{\mu,\,\nu;\,\psi}_{0+}\phi\right\vert^{\kappa(\xi)} d\xi \right)\notag\\
    &\geq& \int_{\Lambda}\frac{1}{\kappa(\xi)} \left\vert^{\rm H}\mathfrak{D}^{\mu,\,\nu;\,\psi}_{0+}\phi_{n}\right\vert^{\kappa(\xi)} d\xi - \int_{\Lambda}  \frac{1}{\kappa(\xi)}  \left\vert^{\rm H}\mathfrak{D}^{\mu,\,\nu;\,\psi}_{0+}\phi\right\vert^{\kappa(\xi)} d\xi. \notag\\
\end{eqnarray}}

{\color{red}Using the inequalities (\ref{(3)})-(\ref{(5)}), it follows that
\begin{equation}\label{(6)}
    \lim_{n\rightarrow+\infty} \int_{\Lambda} \frac{1}{\kappa(\xi)} \left\vert^{\rm H}\mathfrak{D}^{\mu,\,\nu;\,\psi}_{0+}\phi_{n}\right\vert^{\kappa(\xi)} d\xi = \int_{\Lambda} \frac{1}{\kappa(\xi)} \left\vert^{\rm H}\mathfrak{D}^{\mu,\,\nu;\,\psi}_{0+}\phi\right\vert^{\kappa(\xi)}d\xi.
\end{equation}

From Eq.(\ref{(6)}) it follows that the integral of the functions family $\left\{\dfrac{1}{\kappa(\xi)} \left\vert^{\rm H}\mathfrak{D}^{\mu,\,\nu;\,\psi}_{0+}\phi_{n}\right\vert^{\kappa(\xi)} \right\}$ possess absolutely equicontinuity on $\Lambda$. Since
\begin{equation}\label{(7)}
\frac{1}{\kappa(\xi)} \left\vert^{\rm H}\mathfrak{D}^{\mu,\,\nu;\,\psi}_{0+}\phi_{n}\,\,-\,\,^{\rm H}\mathfrak{D}^{\mu,\,\nu;\,\psi}_{0+}\phi \right\vert^{\kappa(\xi)}\leq c\left(\frac{1}{\kappa(\xi)} \left\vert^{\rm H}\mathfrak{D}^{\mu,\,\nu;\,\psi}_{0+}\phi_{n} \right\vert^{\kappa(\xi)}+\frac{1}{\kappa(\xi)} \left\vert^{\rm H}\mathfrak{D}^{\mu,\,\nu;\,\psi}_{0+}\phi\right\vert^{\kappa(\xi)} \right)
\end{equation}
the integrals of the family $\left\{\dfrac{1}{\kappa(\xi)} \left\vert^{\rm H}\mathfrak{D}^{\mu,\,\nu;\,\psi}_{0+}\phi_{n}\,\,-\,\,^{\rm H}\mathfrak{D}^{\mu,\,\nu;\,\psi}_{0+}\phi \right\vert^{\kappa(\xi)} \right\}$ are also absolutely equicontinuous on $\Lambda$ and therefore
\begin{equation}\label{(8)}
    \lim_{n\rightarrow+\infty} \int_{\Lambda} \frac{1}{\kappa(\xi)} \left\vert^{\rm H}\mathfrak{D}^{\mu,\,\nu;\,\psi}_{0+}\phi_{n}\,\,-\,\,^{\rm H}\mathfrak{D}^{\mu,\,\nu;\,\psi}_{0+}\phi \right\vert^{\kappa(\xi)} d\xi=0.
\end{equation}

Using Eq.(\ref{(8)}), one has
\begin{equation}\label{(9)}
    \lim_{n\rightarrow+\infty} \int_{\Lambda}  \left\vert^{\rm H}\mathfrak{D}^{\mu,\,\nu;\,\psi}_{0+}\phi_{n}\,\,-\,\,^{\rm H}\mathfrak{D}^{\mu,\,\nu;\,\psi}_{0+}\phi \right\vert^{\kappa(\xi)} d\xi=0.
\end{equation}

From {\bf Proposition \ref{proposition2.4}} and Eq.(\ref{(9)}), $\phi_{n}\rightarrow \phi$, i.e., $\mathcal{L}^{\mu,\nu}$ is of type $(S_{+})$.

3. By the strictly monotonicity, $\mathcal{L}^{\mu,\nu}$ is an injection. Since
\begin{equation*}
\lim_{||\phi||\rightarrow+\infty} \frac{(\mathcal{L}^{\mu,\nu}\phi,\phi) }{||\phi||}  =\,\, \lim_{||\phi||\rightarrow+\infty}
\dfrac{ \displaystyle\int_{\Lambda}\left\vert \,\,^{\rm H}\mathfrak{D}^{\mu,\,\nu;\,\psi}_{0+}\phi \right\vert^{\kappa(\xi)} d\xi}{||\phi||}=+\infty
\end{equation*}
$\mathcal{L}^{\mu,\nu}$ is coercive, thus $\mathcal{L}^{\mu,\nu}$ is a surjection in view of Minty-Browder theorem \cite{14}. Hence $\mathcal{L}^{\mu,\nu}$ has an inverse mapping
$(\mathcal{L}^{\mu,\nu})^{-1}: \left(\mathcal{H}^{\mu,\nu;\,\psi}_{\kappa(\xi)}(\Lambda)\right)^{*}\rightarrow \mathcal{H}^{\mu,\nu;\,\psi}_{\kappa(\xi)}(\Lambda)$. Therefore, the continuity of $(\mathcal{L}^{\mu,\nu})^{-1}$ is sufficient to ensure $\mathcal{L}^{\mu,\nu}$ to be a homeomorphism. If $f_{n},f \in \mathcal{H}^{\mu,\nu;\,\psi}_{\kappa(\xi)}(\Lambda)$, $f_{n}\rightarrow f$, let $\phi_{n}=(\mathcal{L}^{\mu,\nu})^{-1} (f_{n})$, $\phi=(\mathcal{L}^{\mu,\nu})^{-1} (f)$, then $\mathcal{L}^{\mu,\nu}(\phi_{n})=f_{n}$, $\mathcal{L}^{\mu,\nu}(\phi)=f$. {\color{blue}So $\left\{\phi_{n} \right\}$ is bounded in $\mathcal{H}^{\mu,\nu;\,\psi}_{\kappa(\xi)}(\Lambda)$. Without loss of generality, we can assume that $\phi_{n}\rightharpoonup \phi_{0}$. Since $f_{n}\rightarrow f$, then}
\begin{equation*}
    \lim_{n\rightarrow+\infty} \left( \mathcal{L}^{\mu,\nu}(\phi_{n}) - \mathcal{L}^{\mu,\nu}(\phi_{0}),\phi_{n}-\phi_{0} \right)= \lim_{n\rightarrow+\infty} (f_{n},\phi_{n}-\phi_{0})=0.
\end{equation*}
Since $\mathcal{L}^{\mu,\nu}$ is of type $(S_{+})$, $\phi_{n}\rightarrow \phi_{0}$, we conclude that $\phi_{n}\rightarrow\phi$, so $(\mathcal{L}^{\mu,\nu})^{-1}$ is continuous.}
\end{proof}

\begin{proposition}\label{p3}{\rm \cite{Fan1}} {\rm(Hölder-type inequality)} The conjugate space of $\mathscr{L}^{\kappa(\xi)}(\Lambda)$ is $\mathscr{L}^{q(\xi)}(\Lambda)$ where $\dfrac{1}{q(\xi)}+\dfrac{1}{\kappa(\xi)}=1$. For every $\phi\in \mathscr{L}^{\kappa(\xi)}(\Lambda)$ and $v\in\mathscr{L}^{q(\xi)}(\Lambda),$ it follows that
\begin{equation*}
   \left\vert \int_{\Lambda}\phi vd \xi\right\vert \leq\left(\frac{1}{\kappa^{-}}+\frac{1}{q^{-}}\right)||\phi||_{\kappa(\xi)}||v||_{q(\xi)}.
\end{equation*}

\end{proposition}

\begin{definition} We say that $\phi\in \mathcal{H}^{\mu,\nu;\,\psi}_{\kappa(\xi)}(\Lambda)$ is a weak solution of the problem {\rm(\ref{eq1})}, if
\begin{align}\label{a1}\mathfrak{M}\left(\int_{\Lambda}\frac{1}{\kappa(\xi)}\left\vert^{\rm H}\mathfrak{D}^{\mu,\,\nu;\,\psi}_{0+}\phi\right\vert^{\kappa(\xi)}d \xi\right)\int_{\Lambda}\left\vert^{\rm H}\mathfrak{D}^{\mu,\nu;\,\psi}_{0+}\phi\right\vert^{\kappa(\xi)-2}~^{\rm H}\mathfrak{D}^{\mu,\nu;\,\psi}_{0+}\phi~^{\rm H}\mathfrak{D}^{\mu,\nu;\,\psi}_{0+}v d \xi=\int_{\Lambda}\mathfrak{g}(\xi,\phi)v d \xi
\end{align} 
where $v\in \mathcal{H}^{\mu,\nu;\,\psi}_{\kappa(\xi)}(\Lambda).$
\end{definition}

{\color{red}Define
\begin{align}
	\Phi(\phi)=\widehat{\mathfrak{M}}\left(\int_{\Lambda}\frac{1}{\kappa(\xi)}\left\vert^{\rm H}\mathfrak{D}^{\mu,\nu;\,\psi}_{0+}\phi\right\vert^{\kappa(\xi)}d \xi\right)\,\,\,and\,\,\, \Psi(\phi)=\int_{\Lambda}G(\xi,\phi)d \xi
\end{align} 
where 
\begin{equation*}
    \widehat{\mathfrak{M}}(t)=\int_{0}^{t}\mathfrak{M}(s)ds,~G(\xi,\phi)=\int_{0}^{\phi}\mathfrak{g}(\xi,t)dt.
\end{equation*}}

{\color{blue}The associated energy functional $\mathfrak{E}=\Phi (\phi)-\Psi (\phi):\mathcal{H}^{\mu,\nu;\,\psi}_{\kappa(\xi)}(\Lambda)\rightarrow\mathbb{R}$ to problem (\ref{eq1}) is well} defined. Note that $\mathfrak{E}\in C^{1}\left(\mathcal{H}^{\mu,\nu;\,\psi}_{\kappa(\xi)}(\Lambda),\mathbb{R}\right),$ is a weakly lower semi-continuous and $\phi\in\mathcal{H}^{\mu,\nu;\,\psi}_{\kappa(\xi)}(\Lambda)$ is a weak solution of the problem (\ref{eq1}) if and only if $\phi$ is a critical point of $\mathfrak{E}$. Moreover, yields
\begin{align}\label{novo1}
	\mathfrak{E}'(\phi)v
	&=\mathfrak{M}\left(\int_{\Lambda}\frac{1}{\kappa(\xi)}\left\vert^{\rm H}\mathfrak{D}^{\mu,\nu;\,\psi}_{0+}\phi\right\vert^{\kappa(\xi)}d \xi\right)\int_{\Lambda}\left\vert^{\rm H}\mathfrak{D}^{\mu,\nu;\,\psi}_{0+}\phi\right\vert^{\kappa(\xi)-2}~^{\rm H}\mathfrak{D}^{\mu,\nu;\,\psi}_{0+}\phi~^{\rm H}\mathfrak{D}^{\mu,\nu;\,\psi}_{0+}v d \xi\notag \\
	&\qquad
	-\int_{\Lambda}\mathfrak{g}(\xi,\phi)vd \xi\notag\\
	& =\Phi'(\phi)v -\Psi'(\phi),~~\forall v \in \mathcal{H}^{\mu,\nu;\,\psi}_{\kappa(\xi)}(\Lambda).
\end{align}

\begin{definition} We say that $\mathfrak{E}$ satisfies (PS) condition in $\mathcal{H}^{\mu,\nu;\,\psi}_{\kappa(\xi)}(\Lambda)$ if any sequence $(\phi_{n})\subset \mathcal{H}^{\mu,\nu;\,\psi}_{\kappa(\xi)}(\Lambda)$ such that $\{\mathfrak{E}(\phi_{n})\}$ is bounded and $\mathfrak{E}'(\phi_{n})\rightarrow 0$ as $n\rightarrow+\infty,$ has a convergent subsequence. \end{definition}

\begin{lemma}\label{lm1} If $\mathfrak{M}(t)$ satisfies {\rm($C_{0}$)} and {\rm($C_{1}$)}, $\mathfrak{g}$ satisfies {\rm$(f_{0})$} and ($f_1$), then $\mathfrak{E}$ satisfies {\rm(PS)} condition.
\end{lemma}

\begin{proof}{\color{red} Consider the sequence $\left\{\phi_{n}\right\}$ in $\mathcal{H}^{\mu,\nu;\,\psi}_{\kappa(\xi)}(\Lambda)$, such that $|\mathfrak{E}(\phi_{n})|\leq c$ and $\mathfrak{E}'(\phi_{n})\rightarrow 0$. Hence, one has
	\begin{align}
		& c+||\phi_{n}||\geq \mathfrak{E}(\phi_{n})-\frac{1}{\theta}\mathfrak{E}'(\phi_{n})\phi_{n}\nonumber\\
		& =\widehat{\mathfrak{M}}\left(\int_{\Lambda}\frac{1}{\kappa(\xi)}\left\vert^{\rm H}\mathfrak{D}^{\mu,\nu;\,\psi}_{0+}\phi_{n}\right\vert^{\kappa(\xi)}d \xi\right)-\int_{\Lambda}G(\xi,\phi_{n})d \xi\nonumber\\
		& -\frac{1}{\theta}\mathfrak{M}\left(\int_{\Lambda}\frac{1}{\kappa(\xi)}\left\vert^{\rm H}\mathfrak{D}^{\mu,\nu;\,\psi}_{0+}\phi\right\vert^{\kappa(\xi)}d \xi\right)\int_{\Lambda}\left\vert^{\rm H}\mathfrak{D}^{\mu,\nu;\,\psi}_{0+}\phi_{n}\right\vert^{\kappa(\xi)}d \xi + \int_{\Lambda}\frac{1}{\theta}\mathfrak{g}(\xi,\phi_{n})\phi_{n}d \xi\nonumber\\
		& \geq {\color{blue} (1-\omega)\mathfrak{M}\left(\int_{\Lambda}\frac{1}{\kappa(\xi)}\left\vert^{\rm H}\mathfrak{D}^{\mu,\nu;\,\psi}_{0+}\phi_{n}\right\vert^{\kappa(\xi)}d \xi\right) t-\int_{\Lambda}G(\xi,\phi_{n})d \xi} \nonumber\\
		& {\color{blue}-\frac{1}{\theta}\mathfrak{M}\left(\int_{\Lambda}\frac{1}{\kappa(\xi)}\left\vert^{\rm H}\mathfrak{D}^{\mu,\nu;\,\psi}_{0+}\phi\right\vert^{\kappa(\xi)}d \xi\right)\int_{\Lambda}\left\vert^{\rm H}\mathfrak{D}^{\mu,\nu;\,\psi}_{0+}\phi_{n}\right\vert^{\kappa(\xi)}d\xi + \int_{\Lambda} \frac{1}{\theta} \mathfrak{g}(\xi, \phi_n)\phi_n d\xi}\nonumber\\ 
		& \geq\left(\frac{1-\omega}{\kappa^{+}}-\frac{1}{\theta}\right)\mathfrak{M}\left(\int_{\Lambda}\frac{1}{\kappa(\xi)}\left\vert^{\rm H}\mathfrak{D}^{\mu,\nu;\,\psi}_{0+}\phi_{n}\right\vert^{\kappa(\xi)}d \xi\right)\int_{\Lambda}\left\vert^{\rm H}\mathfrak{D}^{\mu,\nu;\,\psi}_{0+}\phi_{n}\right\vert^{\kappa(\xi)}d \xi \nonumber\\
		&\qquad +\int_{\Lambda}\left(\frac{1}{\theta}\mathfrak{g}(\xi,\phi_{n})\phi_{n}-G(\xi,\phi_{n})\right)d \xi\nonumber\\
		& \geq\left(\frac{1-\omega}{\kappa^{+}}-\frac{1}{\theta}\right)m_{0}\int_{\Lambda}\left\vert^{\rm H}\mathfrak{D}^{\mu,\nu;\,\psi}_{0+}\phi_{n}\right\vert^{\kappa(\xi)}d \xi-c \nonumber\\
		& \geq\left(\frac{1-\omega}{\kappa^{+}}-\frac{1}{\theta}\right)m_{0}\left\|\phi_{n}\right\|^{\kappa^{-}}-c.
	\end{align}}

So $\left\{\left\|\phi_{n}\right\|_{\kappa(\xi)}\right\}$ is bounded. We assume that $\phi_{n}\rightharpoonup \phi$ (without loss of generality), {\color{blue}then $\mathfrak{E}'(\phi_{n})(\phi_{n}-\phi)\rightarrow 0$.} {\color{red}Thus, one has
	\begin{align}\label{eq3.10}
	& \mathfrak{E}'(\phi_{n})(\phi_{n}-\phi) \nonumber\\
		&=\mathfrak{M}\left(\int_{\Lambda}\frac{1}{\kappa(\xi)}\left\vert^{\rm H}\mathfrak{D}^{\mu,\nu;\,\psi}_{0+}\phi_{n}\right\vert^{\kappa(\xi)}d \xi\right) \int_{\Lambda}\left\vert^{\rm H}\mathfrak{D}^{\mu,\nu;\,\psi}_{0+}\phi_{n}\right\vert^{\kappa(\xi)-2}~^{\rm H}\mathfrak{D}^{\mu,\nu;\,\psi}_{0+}\phi_{n}~^{\rm H}\mathfrak{D}^{\mu,\nu;\,\psi}_{0+}(\phi_{n}-\phi)d \xi\nonumber\\
		&-\int_{\Lambda}G(\xi,\phi_{n})(\phi_{n}-\phi)\nonumber\\
		& =\mathfrak{M}\left(\int_{\Lambda}\frac{1}{\kappa(\xi)}\left\vert^{\rm H}\mathfrak{D}^{\mu,\nu;\,\psi}_{0+}\phi_{n}\right\vert^{\kappa(\xi)}d \xi\right) \int_{\Lambda}\left\vert^{\rm H}\mathfrak{D}^{\mu,\nu;\,\psi}_{0+}\phi_{n}\right\vert^{\kappa(\xi)-2}~^{\rm H}\mathfrak{D}^{\mu,\nu;\,\psi}_{0+}\phi_{n}\, \left(^{\rm H}\mathfrak{D}^{\mu,\nu;\,\psi}_{0+}\phi_{n}-\,\,^{\rm H}\mathfrak{D}^{\mu,\nu;\,\psi}_{0+}\phi \right)d \xi\nonumber\\
		&-\int_{\Lambda} \mathfrak{g}(\xi,\phi_{n})(\phi_{n}-\phi) d \xi \rightarrow 0.\notag\\
	\end{align}}
	
Using the condition $(f_{0})$, {\bf Proposition \ref{p1}} and {\bf Proposition \ref{p3}} {\color{blue} it follows that} \begin{equation*}
    \int_{\Lambda}\mathfrak{g}(\xi,\phi_{n})(\phi_{n}-\phi)d \xi\rightarrow 0.
\end{equation*}

{\color{red}In this sense, we obtain
\begin{align*}
\mathfrak{M}\int_{\Lambda} \left(\frac{1}{\kappa(\xi)} \left\vert ^H\mathfrak{D}_{0+}^{\mu,\nu;\,\psi} \phi_n \right\vert ^{\kappa(\xi)} d \xi \right) \int_{\Lambda}\left\vert^H\mathfrak{D}_{0+}^{\mu,\nu;\,\psi}\phi_n \right\vert ^{\kappa(\xi)-2}\, ^H\mathfrak{D}_+^{\mu,\nu;\,\psi}  \left(^H\mathfrak{D}_+^{\mu,\nu;\,\psi}\phi_n -\,\, ^H\mathfrak{D}_+^{\mu,\nu;\,\psi}\phi \right) d \xi \to 0.
\end{align*}}

Using the condition $(C_0)$, yields
\begin{align*}	
\int_{\Lambda} \left\vert ^H\mathfrak{D}_{0+}^{\mu,\nu;\,\psi}\right\vert^{\kappa(\xi)-2} {^H\mathfrak{D}_{0+}^{\mu,\nu;\,\psi} }\phi_n \left(^H\mathfrak{D}_+^{\mu,\nu;\,\psi} \phi_n -\,\,^H\mathfrak{D}_+^{\mu,\nu;\,\psi}\phi\right)d \xi \to 0 
\end{align*}

Finally, using {\bf Proposition \ref{p2}}, hence $\phi_n \to \phi$. Therefore, we complete the proof.
\end{proof}

Since $\mathcal{H}^{\mu,\nu;\,\psi}_{\kappa(\xi)}(\Lambda)$ is a reflexive and separable Banach space (see {\bf Proposition \ref{bana}}), there exist $\{\epsilon_j\}\subset \mathcal{H}^{\mu,\nu;\,\psi}_{\kappa(\xi)}(\Lambda)$ and $\{\epsilon _j^*\} \subset \left(\mathcal{H}^{\mu,\nu;\,\psi}_{\kappa(\xi)}(\Lambda)\right)^* $ such that $ \mathcal{H}^{\mu,\nu;\,\psi}_{\kappa(\xi)}(\Lambda)=\overline{\mbox{span}\{\epsilon_j: j=1,2,,...\}},~~ \left(\mathcal{H}^{\mu,\nu;\,\psi}_{\kappa(\xi)}(\Lambda)\right)^*=\overline{\mbox{span}\{\epsilon_j^*:j=1,2,3,...\}}$ and
\begin{align*}
<\epsilon_j,\epsilon_j^*>=
 \begin{cases}
 1~~\text{if}~~ i=j\\
 0~~\text{if}~~ i\neq j
\end{cases}
 \end{align*}
 
{\color{blue}Let's use $\left(\mathcal{H}^{\mu,\nu;\,\psi}_{\kappa(\xi)}(\Lambda)\right)_j =\mbox{span}\{\epsilon_j\}$, $Y_k =\oplus_{j=1}^{k} \left(\mathcal{H}^{\mu,\nu;\,\psi}_{\kappa(\xi)}(\Lambda)\right)_j$ and $Z_k =\overline{\oplus_{j=k}^{+\infty} \left(\mathcal{H}^{\mu,\nu;\,\psi}_{\kappa(\xi)}(\Lambda)\right)_j}$.}

\begin{lemma}{\rm \cite{Fan1}} {\color{blue} If $\mu \in C_+(\bar{\Lambda})$, $\mu(\xi)<\kappa_{\mu}^{*} (\xi)$ for any $\xi\in \bar{\Lambda}$, denote
	\begin{equation}
		\beta_{k} =\sup \{|\phi|_{\mu(\xi)}: \vert|\phi \vert| =1, \phi \in Z_k\}
	\end{equation}
then  $\lim_{k\to +\infty} \beta_k =0$.}
\end{lemma}

\begin{lemma}\label{foun}{\rm\cite{Willem}} {\rm (Fountain Theorem)} Assume 
\begin{itemize}
\item [($A_1$)] $X$ is a Banach Space, $\mathfrak{E} \in C^1(X,\mathbb{R})$ is an even functional. 

If for each $k=1,2,...$ there exist $\rho_{k}> r_{k} >0$ such that:

\item [($A_2$)] $\inf_{\phi\in Z_k, \vert|\phi \vert|=r_k} \mathfrak{E}(\phi) \to +\infty$ as $k \to +\infty$.

\item [($A_3$)] {\color{blue}$\max_{\phi\in Y_k, ||\phi||=\rho_{k}} \mathfrak{E}(\phi) \leq 0$.}

\item [($A_4$)] $\mathfrak{E}$ satisfies {\rm(PS)} condition for every $c>0$,
\end{itemize}
then $\mathfrak{E}$ has a sequence of critical values {\color{blue}tending} to $+\infty$.
\end{lemma}

\begin{lemma}\label{dual}{\rm\cite{Willem}} {\rm(Dual Fountain Theorem)} Assume $(A_{1})$ is satisfied and there is  $k_0>0$ so that for each $k\geq k_{0}$, there exist $\rho_{k}> \gamma_k >0$ such that 
\begin{itemize}

\item [($B_1$)] $\inf_{\phi\in Z_k,\vert|\phi \vert|=\rho_{k}}\mathfrak{E}(\phi)\geq 0$.

\item [($B_2$)] $b_k =\max_{\phi\in Y_k,\vert|\phi \vert|=r_k} \mathfrak{E}(\phi) <0$.

\item [($B_3$)] $d_k =\inf_{\phi\in Z_k,\vert|\phi \vert|\leq\rho_{k}} \mathfrak{E}(\phi) \to 0$ as $k \to +\infty$.

\item [($B_4$)] $\mathfrak{E}$ satisfies $(PS)_c^*$ condition for {\color{blue}every $c\in [d_{k_0}, 0)$.}
\end{itemize}

Then $\mathfrak{E}$ has a sequence of negative critical values converging to 0. 
\end{lemma}

\begin{definition} We say that $\mathfrak{E}$ satisfies this $(PS)_c^*$ condition with respect to $(Y_n)$, if any sequence $\{\phi_{n_j}\}\subset \mathcal{H}^{\mu,\nu;\,\psi}_{\kappa(\xi)}(\Lambda) $ such that $n_j \to +\infty$, $\phi_{n_j}\in Y_n$, $\mathfrak{E}(\phi_{n_J}) \to c$ and $(\mathfrak{E}|_{Y_{n_j}})'(\phi_{n_j})\to 0$, contains a subsequence converging to a critical point of $\mathfrak{E}$.
\end{definition}

\begin{lemma}{\rm\cite{Dai}}\label{lm3} Assume that ${\rm (C_{0})}$, ${\rm (C_{1})}$,  ${\rm (f_0)}$ and ${\rm (f_1)}$ hold, then $\mathfrak{E}$ satisfies the $(PS)_{c}^{*}$ condition.
\end{lemma}

%%%%%%%%%%%%%%%%%%%%%%%%%%%%%%%%%%%%%%%%%%%%%%%%%
\section{Main results}

In this section, we will address the main results of this paper, using variational techniques and results from Sobolev spaces with variable exponents and from the $\psi$-fractional space $\mathcal{H}^{\mu,\nu;\,\psi}_{\kappa(\xi)}(\Lambda)$, as discussed in the previous section.

\begin{proof} {\bf(Proof of Theorem \ref{Teorema3.1})} {\color{blue}Using the inequalities \eqref{ineq1} and \eqref{ineq2}, yields
\begin{align}\label{ineq3}
\left|G(\xi,t)\right|&=\left\vert\int_{0}^{t}\mathfrak{g}(\xi,s)ds\right\vert\leq \int_{0}^{t}|\mathfrak{g}(\xi,s)ds\nonumber\\
&\leq c\int_{0}^{t}(1+|s|^{\widehat{\beta}-1})ds\nonumber\\&=
c\int_{0}^{t}ds+c\int_{0}^{t}|s|^{\widehat{\beta}-1}ds\nonumber\\&\leq  c\left(|t|+|t|^{\widehat{\beta}}\right)
\end{align}
and $\widehat{\mathfrak{M}}(t)\geq m_{0} t$. In this sense follows of (\ref{ineq3})
\begin{align*}
		\mathfrak{E}(\phi)&=\widehat{\mathfrak{M}} \left(\int_{\Lambda}\frac{1}{\kappa(\xi)}\left\vert^{\rm H}\mathfrak{D}^{\mu,\nu;\,\psi}_{0+}\phi\right\vert^{\kappa(\xi)}d \xi\right)-\int_{\Lambda}G(\xi,\phi)d \xi\\
		&\geq m_{0}\int_{\Lambda}\frac{1}{\kappa(\xi)}\left\vert^{\rm H}\mathfrak{D}^{\mu,\nu;\,\psi}_{0+}\phi\right\vert^{\kappa(\xi)}d \xi-c\int_{\Lambda}|\phi|^{\widehat{\beta}}d \xi-c\int_{\Lambda}|\phi|d \xi\\
		&\geq\frac{m_{0}}{\kappa^{+}}||\phi||^{\kappa^{-}}-c\left\|\phi\right\|^{\widehat{\beta}}-c\left\|\phi\right\|\rightarrow +\infty,
\end{align*}
as $||\phi||\rightarrow +\infty$. Since $\mathfrak{E}$ is weakly lower semi-continuous, $\mathfrak{E}$ has a minimum point $\phi$ in $\mathcal{H}^{\mu,\nu;\,\psi}_{\kappa(\xi)}(\Lambda),$ and $\phi$ is a weak solution of problem \eqref{eq1}.}
\end{proof}

{\color{red}\begin{proof} {\bf (Proof of Theorem \ref{Teorema3.2})} Using {\bf Lemma \ref{lm1}}, $\mathfrak{E}$ satisfies (PS) condition in $\mathcal{H}^{\mu,\nu;\,\psi}_{\kappa(\xi)}(\Lambda)$. Since $ \kappa^{+} < \zeta^{-} \leq 
	\zeta(\xi) < \kappa_{\mu}^{*}(t)$, $\mathcal{H}^{\mu,\nu;\,\psi}_{\kappa(\xi)}(\Lambda) \hookrightarrow 
	\mathscr{L}^{\kappa^{+}}(\Lambda)$ then there exist $c>0$ such that
	\begin{equation}\label{eq3.11}
		|\phi|_{\kappa^{+}} \leq c \lVert \phi \rVert, \phi\in \mathcal{H}^{\mu,\nu;\,\psi}_{\kappa(\xi)}(\Lambda).
	\end{equation}
	
Let $\epsilon >0$ such that $\epsilon c^{\kappa^{+}} \leq \dfrac{m_0}{2\kappa^{+}}$. From the conditions $(f_0)$ and $(f_2)$ , yields	
\begin{equation}\label{ineq4}
		G(\xi,t) \leq \epsilon |t|^{\kappa^{+}} + c|t|^{\zeta(\xi)} ,(\xi,t)\in \Lambda \times \mathbb{R}.
\end{equation}

Using the condition $(C_{0})$ and the inequality (\ref{ineq4}), yields
	\begin{align}
		\mathfrak{E}(\phi) 
		&\geq \frac{m_0}{\kappa^{+}} \int_{\Lambda} \left\vert^H \mathfrak{D}_{0+}^{\mu ,\beta, \psi} \phi\right\vert^{\kappa(\xi)} d \xi - \epsilon \int_{\Lambda} |\phi|^{\kappa^{+}} d \xi - c \int_{\Lambda} |\phi|^{\zeta(\xi)} d \xi \nonumber\\
		&\geq \frac{m_0}{\kappa^{+}} \vert\vert \phi \vert\vert^{\kappa^{+}} - \epsilon c^{\kappa+} \vert\vert \phi \vert\vert^{\kappa^{+}} - c\vert\vert \phi \vert\vert^{\zeta^{-}}\nonumber\\
		& \geq \frac{m_0}{2\kappa^{+}} \vert\vert \phi \vert\vert^{\zeta^{+}}-c\vert\vert \phi \vert\vert^{\zeta^{-}},
		&
	\end{align}
when $||\phi||\leq 1$.
	
Therefore, there exists $r>0 ,\delta >0$ such that, $\mathfrak{E}(\phi)\geq\delta>0$ for every $\vert\vert \phi \vert\vert =r$. From $(f_1)$, it follows that 
	\begin{align*}
		G(\xi,t) \geq c \vert t\vert ^{\theta}, ~\xi \in \Lambda , ~|t| \geq T.
	\end{align*} 
{\color{blue}
Consider the conditions {\rm($C_{0}$)} and {\rm($C_{1}$)}. Note that the function $g(t)=\dfrac{\widehat{\mathfrak{M}}(t)}{t^{1/w-1}}$ is decreasing. So for all $t_{0}> 0$, when $t>t_0$, yields
\begin{eqnarray*}
    g(t)\leq g(t_0).
\end{eqnarray*}
In this sense, from $\dfrac{\widehat{\mathfrak{M}}(t)}{t^{1/w-1}}\leq \dfrac{\widehat{\mathfrak{M}}(t_0)}{t^{1/w-1}}$, it follows that $\ln (\widehat{\mathfrak{M}}(t))\leq \ln (\widehat{\mathfrak{M}}(t_0))-\dfrac{1}{1-w}\ln t-\dfrac{1}{1-w}\ln t_0$. Therefore, one has
\begin{equation}
		\widehat{\mathfrak{M}}(t) \leq \dfrac{\widehat{\mathfrak{M}}(t_0)}{t_{0}^{1/1-w}} t^{1/1-\omega} \leq ct^{\frac{1}{1-\omega}},\,\, for\,\, t>t_{0}
	\end{equation}
where $t_0>0$ (constant).} For $w \in \mathcal{H}^{\mu,\nu;\,\psi}_{\kappa(\xi)}(\Lambda)-\{0\}$ and $t>1$, yields
	\begin{align*}
		\mathfrak{E}(t w)
		&= \widehat{\mathfrak{M}} \left(\int_{\Lambda} \frac{1}{\kappa(\xi)}\left\vert t\, ^H\mathfrak{D}_{0+}^{\mu,\nu;\,\psi} w \right\vert ^{\kappa(\xi)} d \xi\right) - \int _{\Lambda} G(x,tw)d \xi\\
		& \leq ct^{\frac{\kappa^{+}}{1-\omega}} \left(\int_{\Lambda} \frac{1}{\kappa(\xi)}\left\vert  ^H\mathfrak{D}_{0+}^{\mu,\nu;\,\psi} w \right\vert ^{\kappa(\xi)} d \xi\right) ^{\frac{1}{1-\omega}} - ct^{\theta} \int_{\Lambda} |w|^{\theta} d \xi -c\\
		&	\to -\infty ~~\mbox{as}~~ t \to +\infty.
	\end{align*}
due to $\theta > \dfrac{\kappa^{+}}{1-\omega}$. Since $\mathfrak{E}(0)=0$, $\mathfrak{E}$ satisfies the conditions of the Mountain Pass Theorem. So $\mathfrak{E}$ admits at least one nontrivial critical point.
\end{proof}}

Now, we will use the {\bf Lemma \ref{foun}} (Fountain Theorem) and {\bf Lemma \ref{dual}} (Dual Fountain Theorem) to prove {\color{blue}{\bf Theorem \ref{Teorema3.3}} and {\bf Theorem \ref{Teorema3.4}},} respectively.

\begin{proof} {\bf (Proof of Theorem \ref{Teorema3.3})} Note that $\mathfrak{E}$ is an even function and satisfies (PS) {\color{blue}condition }(see condition $(f_3)$ and {\bf Lemma \ref{lm1}}). Purpose here is proof that there is $\rho_{k} >\gamma _k >0$ ($k$ large) such that ({\rm $A_2$}) and (${\rm A_3}$) hold and, so use {\bf Lemma \ref{foun}} {\color{blue}(Fountain Theorem)}.

{\color{red}
{\rm ($A_2$)} For any $\phi \in Z_k$, $\eta\in\Lambda$, $\vert|\phi \vert|$ = $\gamma_k$ = $\left(c\zeta^{+}\beta_{k}^{\zeta^+}m_{0}^{-1}\right)^{\frac{1}{\kappa^{-}-\zeta^+}}$, it follows that
\begin{align*}
\mathfrak{E}(\phi)
&= \widehat{\mathfrak{M}} \left(\int_{\Lambda} \frac{1}{\kappa(\xi)}\left\vert ^{\rm H}\mathfrak{D}_{0+}^{\mu,\nu;\,\psi} \phi \right\vert ^{\kappa(\xi)} d \xi\right) -\int_{\Lambda} \mathfrak{g}(\xi,\phi)d \xi\\
& \geq \frac{m_0}{\kappa^{+}} \int_{\Lambda} \frac{1}{\kappa(\xi)}\left\vert ^{\rm H}\mathfrak{D}_{0+}^{\mu,\nu;\,\psi} \phi \right\vert ^{\kappa(\xi)} d \xi - c\int_{\Lambda} |\phi|^{\zeta(\xi)}d \xi -c\int_{\Lambda} |\phi|d \xi \\
& \geq \frac{m_0}{\kappa^{+}} \vert|\phi \vert|^{\kappa^{-}} - c \vert|\phi \vert|^{\zeta(\eta)} -c \vert|\phi \vert|\\
& \geq \vert|\phi \vert|^{\kappa^{-}} -c\beta_k^{\zeta^+} \vert|\phi \vert|^{\kappa^{+}} -c\vert|\phi \vert|-c \\
&= m_0 \left(\frac{1}{\kappa^{+}}-\frac{1}{\zeta^+}\right) \left(c\zeta^+\beta_\kappa^{\zeta^+}m_0^{-1}\right)^{\frac{\kappa^{+}}{\kappa^{+} -\zeta^+}} -c\left(c\zeta^+\beta_\kappa^{\zeta^+}m_0^{-1}\right)^{\frac{1}{\kappa^{-} -\zeta^+}} -c \to +\infty 
\end{align*}
as $\kappa\rightarrow+\infty$ and with $\kappa^{+} <\zeta^+, \kappa^{-} >1 $ and $\beta_{k} \to 0$.}

{\color{red}{\rm ($A_3$)} Using ${\rm(f_1)}$, we have $G(\xi,t) \geq c |t|^{\theta} -c$. Therefore, for any $w\in Y_k$ with $\vert|w \vert|= 1$ and $1< t =\rho_{k}$, yields
\begin{align}
\mathfrak{E}(t w)
&= \widehat{\mathfrak{M}} \left(\int_{\Lambda} \frac{1}{\kappa(\xi)}\left\vert t\,\, ^H\mathfrak{D}_{0+}^{\mu,\nu;\,\psi} w \right\vert ^{\kappa(\xi)} d \xi\right) - \int _{\Lambda} G(\xi,t w)d \xi\nonumber\\
& \leq c\rho_{k}^{\frac{\kappa^{+}}{1-\omega}} \left(\int_{\Lambda} \left\vert ^H\mathfrak{D}_{0+}^{\mu,\nu;\,\psi} \phi \right\vert ^{\kappa(\xi)} d \xi\right) ^{\frac{1}{1-\omega}} - c\rho_{k}^{\theta} \int_{\Lambda} |w|^{\theta}d \xi -c.
&
\end{align}}

Note that, since $\theta > \dfrac{\kappa^{+}}{1-\omega}$ and dim $Y_k = k$ holds, $\mathfrak{E}(\phi) \to -\infty$ as $\vert|\phi \vert| \to +\infty$ for $\phi \in Y_k$. In this sense, using the {\bf Lemma \ref{dual}} (Fountain theorem), we concluded the proof of Theorem.
\end{proof}

\begin{proof}{\bf (Proof of Theorem \ref{Teorema3.4})} First, note that, using condition $(f_3)$ and {\bf Lemma \ref{lm3}}, it follows that $\mathfrak{E}$ satisfies the conditions $(A_{1})$ and $(B_{4})$ (see {\bf Lemma \ref{dual}}-Dual Fountain Theorem).

{\color{red}
\begin{itemize}
\item [($B_1$)]  For any $v\in Z_k$, $\vert|v\vert|=1$, and $0<t<1$, yields
\end{itemize}}{\color{red}
\begin{align}\label{ineq5}
\mathfrak{E}(tv)
&= \widehat{\mathfrak{M}} \left(\int_{\Lambda} \frac{1}{\kappa(\xi)}\left\vert  ^H\mathfrak{D}_{0+}^{\mu,\nu;\,\psi} (tv) \right\vert ^{\kappa(\xi)} d \xi\right) - \int _{\Lambda} G(\xi,tv)d \xi\nonumber\\
& \geq \frac{m_0}{\kappa^{+}} t^{\kappa^{+}} \int_{\Lambda} \left\vert ^H\mathfrak{D}_{0+}^{\mu,\nu;\,\psi} v \right\vert ^{\kappa(\xi)} d \xi -\epsilon t^{\kappa^{+}} \int_{\Lambda} |v|^{\kappa^{+}} d \xi -c t^{\zeta^-} \int_{\Lambda} |v|^{\zeta(\xi)} d \xi\nonumber\\
& \geq  \frac{m_0}{2\kappa^{+}} t^{\kappa^{+}} 
- \begin{cases}
 c ~{\beta_{k}^{\zeta^-}~t^{\zeta^{-}}}~\text{if}~~ \|\phi\|_{\zeta(\xi)\leq 1}\\
 c ~{\beta_{k}^{\zeta^+}~t^{\zeta^{-}}}~\text{if}~~ \|\phi\|_{\zeta(\xi)> 1}.
\end{cases}
\end{align}}

{\color{red}
Since $\zeta^- > \kappa^{+}$, without loss of generality taking $\rho_{k} =t$ with $k$ (sufficiently large), for $v \in Z_k$ with $\vert| v\vert|= 1$, holds $\mathfrak{E}(tv) \geq 0$. In that sense, we have $\mathop{\inf }_{\begin{subarray}{c}\phi\in Z_k,\vert|\phi \vert|=\rho_{k} \end{subarray}} \mathfrak{E}(\phi) \geq 0$ for $k$ sufficiently large. Thus, the condition ($B_1$) is satisfied.}

\begin{itemize}
\item [($B_2$)] For $v \in Y_k$, $\|v\|=1$ and $0<t<\rho_{k}<1$, yields
\end{itemize}{\color{red}
\begin{align*}
\mathfrak{E}(tv)& = \widehat{\mathfrak{M}} \left( \int_{\Lambda} \frac{1}{\kappa(\xi)}\left\vert t~ ^H\mathfrak{D} _{0+}^{\mu,\nu;\,\psi} v\right\vert^{\kappa(\xi)} d \xi \right) -\int_{\Lambda}G(\xi,tv) d \xi \\
& \leq c \left( \int_{\Lambda} \left\vert t~ ^H\mathfrak{D}_{0+}^{\mu,\nu;\,\psi} v\right\vert^{\kappa(\xi)} d \xi \right)^{\frac{1}{1-\omega}} -c \int_{\Lambda}\vert tv\vert ^{\gamma^{(\xi)}} d \xi \\
& \leq c t^{\frac{\kappa^{-}}{1-\omega}} \left( \int_{\Lambda} \left\vert ^H\mathfrak{D}_{0+}^{\mu,\nu;\,\psi} v\right\vert^{\kappa(\xi)} d \xi \right)^{\frac{1}{1-\omega}} -c t^{\gamma^+} \int_{\Lambda} \vert v\vert ^{\gamma(\xi)} d \xi.
	\end{align*}}

Using the fact $\gamma^+ <\dfrac{\kappa^{-}}{1-\omega}$, there exists a $r_k \in (0,\rho_{k})$ such that $\mathfrak{E}(tv)<0$ when $t=r_{k}$. So, we obtain
\begin{equation*}
b_k:=\mathop{\max }_{\begin{subarray}{c}\phi\in Y_k,\\ \vert|\phi \vert|=r_k \end{subarray}} \mathfrak{E}(\phi) < 0.
\end{equation*}
Thus, the condition $(B_{2})$ is satisfied.

{\rm($B_3$)} Using the fact that $Y_k \cap Z_k \neq \emptyset$ and $r_k <\rho_{k}$, one has
\begin{equation*}
d_k:=\mathop{\inf }_{\begin{subarray}{c}\phi\in Z_k,\\ \vert|\phi \vert|\leq \rho_{k} \end{subarray}} \mathfrak{E}(\phi) \leq b_k:=\mathop{\max }_{\begin{subarray}{c}\phi\in Y_k,\\ \vert|\phi \vert|=r_k \end{subarray}} \mathfrak{E}(\phi) <0.
\end{equation*}

Using the inequality \eqref{ineq5}, for $v\in Z_k$, $\vert|v\vert| =1$, $0\leq t\leq \rho_{k}$ and $\phi=tv$, yields
\begin{align*}
\mathfrak{E}(\phi)=\mathfrak{E}(tv) \geq -
\begin{cases}
 c ~{\beta_{k}^{\zeta^-}~t^{\zeta^{-}}}~\text{if}~~ \|\phi\|_{\zeta(\xi)}\leq 1\\
 c ~{\beta_{k}^{\zeta^+}~t^{\zeta^{-}}}~\text{if}~~ \|\phi\|_{\zeta(\xi)}> 1
\end{cases}
\end{align*}
hence $d_k \to 0$ i.e. ($B_3$) is satisfied. Therefore, by means of {\bf Theorem \ref{Teorema3.4}}, we conclude the proof.
\end{proof} 

%%%%%%%%%%%%%%%%%%%%%%%%%%%%%%%%%%%%%%%%%%%%%%%%%%%%%%%%%%%%%%%%%%%%%%%%%%%%%%%%%%%%%%%%%%%%%%%%%%%%%%%%%%%%%%%%%%%%%%
 \section{A special problem and comments}

{\color{red}The following idea is to discuss some consequences of {\bf Theorem \ref{Teorema3.1} - Theorem \ref{Teorema3.4}}. Consider the following fractional problem}
\begin{align}\label{eq11}
\left(a+b\int_{\Lambda}\frac{1}{\kappa(\xi)} \left\vert^{\rm H}\mathfrak{D}^{\mu,\nu;\,\psi}_{0+}\phi\right\vert^{\kappa(\xi)} d \xi\right){\bf L}^{\mu,\nu;\,\psi}_{\kappa(\xi)}\phi&=\mathfrak{g}(\xi,\phi),\,\, in\,\,\Lambda =[0,T]\times [0,T]\notag\\
\phi&=0
\end{align} 
where $${\bf L}^{\mu,\nu;\,\psi}_{\kappa(\xi)}\phi=\,\,^{\rm H}\mathfrak{D}^{\mu,\nu;\,\psi}_{T}\left(\left\vert^{\rm H}\mathfrak{D}^{\mu,\nu;\,\psi}_{0+}\phi\right\vert^{\kappa(\xi)-2}~{^{\rm H}\mathfrak{D}^{\mu,\nu;\,\psi}_{0+}}\phi\right)$$ and $a,b$ are two positive constants.

Let $\mathfrak{M}(t)=a+bt$ with $t=\displaystyle\int_{\Lambda}\frac{1}{\kappa(\xi)} \left\vert^{\rm H}\mathfrak{D}^{\mu,\nu;\,\psi}_{0+}\phi\right\vert^{\kappa(\xi)}d \xi$. Note that, $\mathfrak{M}(t)\geq a>0$ and taking $\omega=\dfrac{4}{5}$, yields
\begin{equation*}
  \widehat{\mathfrak{M}}(t)=\int_{0}^{t} \mathfrak{M}(s)ds={\color{blue}at+\frac{b}{2}t^{2}\geq \frac{
  1}{2}\left(a+bt \right)t \geq \frac{
  1}{5}\left(a+bt \right)t}=(1-\omega) \mathfrak{M}(t)t.
\end{equation*}

Therefore, the conditions $(C_{0})$ and $(C_{1})$ are satisfied. In this sense, as a consequence of {\bf Theorem \ref{Teorema3.1}, {\color{red}Theorem \ref{Teorema3.2}, Theorem \ref{Teorema3.3}} and {\bf Theorem \ref{Teorema3.4}}, we have the following corollaries, namely:}

\begin{corollary} If $\mathfrak{M}$ satisfies $(C_{0})$ and $|\mathfrak{g}(\xi,t)|\leq \mathcal{A}_{1}+\mathcal{A}_{2} |t|^{\beta-1}$, where $1\leq \beta <\kappa^{-}$, then problem {\rm(\ref{eq11})} has a weak solution.
\end{corollary}

\begin{corollary} If $\mathfrak{M}$ satisfies $(C_{0})$ and $(C_{1})$, and $f$ satisfies $(f_{0})$, $(f_{1})$ and $(f_{2})$, where $\zeta^{-}>\kappa^{+}$, then problem {\rm(\ref{eq11})} has a nontrivial solution.
\end{corollary}

\begin{corollary} Assume that the conditions $(C_{0})$, $(C_{1})$, $(f_{0})$, $(f_{1})$ and $(f_{3})$ hold. Then, problem {\rm(\ref{eq11})} has a sequence of solutions $\left\{\pm \phi_{k} \right\}_{k=1}^{+\infty}$ such that $\mathfrak{E}(\pm \phi_{k})\rightarrow +\infty$ as $k\rightarrow +\infty$.
\end{corollary}

\begin{corollary} Assume that the conditions $(C_{0})$, $(C_{1})$, $(f_{0})$, $(f_{1})$, $(f_{3})$ and $(f_{4})$ hold. Then, problem {\rm(\ref{eq11})} has a sequence of solutions $\left\{\pm \phi_{k} \right\}_{k=1}^{+\infty}$ such that $\mathfrak{E}(\pm \phi_{k})\rightarrow +\infty$ as $k\rightarrow +\infty$.
\end{corollary}

\begin{remark} Note that, we can take other functions with respect to $\mathfrak{M}(t)$ and $t$ in {\rm Eq.(\ref{eq11})} and get other versions of Kirchhoff-type problems, {\color{blue}that is:}
\begin{itemize}
    \item $\mathfrak{M}(t)=a+bt$ with $t=\displaystyle\int_{\Lambda}\frac{1}{\kappa} \left\vert^{\rm H}\mathfrak{D}^{\mu,\nu;\,\psi}_{0+}\phi\right\vert^{\kappa}d \xi$.
    
    \item $\mathfrak{M}(t)=bt$ with $t=\displaystyle\int_{\Lambda}\frac{1}{\kappa(\xi)} \left\vert^{\rm H}\mathfrak{D}^{\mu,\nu;\,\psi}_{0+}\phi\right\vert^{\kappa(\xi)}d \xi$. Note that it also holds for $\kappa(\xi)=\kappa$.
    
    \item Note that, we only discuss the special cases above, starting from the particular choice of $\mathfrak{M}(t)=bt$, $t$ and $\kappa(\xi)$. However, it is also possible to obtain and discuss other special cases, through the limits of $\beta\rightarrow 0$, $\beta \rightarrow 1$ and the function $\psi(\cdot)$.
\end{itemize}
\end{remark}

Kirchhoff-type problems are of great interest, in particular, in recent years an approach involving fractional operators has gained prominence. After the results investigated above, some future questions can be addressed, namely:
\begin{itemize}
    \item Discuss the same objectives of the present article for Kirchhoff-type problems with double phase.
    
     \item Another investigation possibility is to modify the problem boundary condition (\ref{eq1}), to Neumann boundary.
\end{itemize}

\section*{Acknowledgements}

{\color{blue} The authors thank very grateful to the anonymous reviewers for their useful comments that led to improvement of the manuscript. Juan J. Nieto thanks the Agencia Estatal de Investigación (AEI) of Spain under Grant PID2020-113275GB-I00, cofinanced by the European Community fund FEDER, as
as well as Xunta de Galicia grant ED431C 2019/02 for Competitive Reference Research Groups (2019-22).}
 %%%%%%%%%%%%%%%%%%5%%%%%%%%%%%%%%%%%%%%%%%%%%%%%%%%%%%%%%%%%%%%%%%%%%%%%%%%%%%%%%%%%

%%%%%%%%%%%%%%%%%%%%%%%%%%%%%%%%%%%%

\end{document}